\documentclass[reqno]{amsart}

\usepackage{amscd, amsmath}
\usepackage{hyperref}
\usepackage{amsmath,amsfonts,amsthm,amssymb,graphics}
\usepackage{paralist}
\usepackage{amssymb,latexsym}
\usepackage{amsmath}
\usepackage{eepic}

\newcommand{\g}{\geqslant}

\newcommand{\RR}{\mathbb{R}}
\newcommand{\ZZ}{\mathbb{Z}}
\newcommand{\CC}{\mathbb{C}}

\newcommand{\NN}{\mathbb{N}}
\newcommand{\p}{\partial}

\newcommand{\les}{\leqslant}
\newcommand{\lesa}{\lesssim}

\newcommand{\supp}{\text{supp }\,}

\theoremstyle{plain}
\newtheorem{theorem}{Theorem}
\newtheorem{proposition}[theorem]{Proposition}
\newtheorem{lemma}[theorem]{Lemma}

\newtheorem{corollary}[theorem]{Corollary}

\setdefaultenum{(i)}{(a)}{1}{A}
\title[Global existence for a nonlinear Dirac Equation]{Global existence for an $L^2$ critical nonlinear Dirac equation in one dimension }
\author[Timothy Candy]{Timothy Candy \vspace{0.4cm}\\
\textit{D\lowercase{epartment} \lowercase{of} M\lowercase{athematics}, U\lowercase{niversity} \lowercase{of} E\lowercase{dinburgh}\\
E\lowercase{dinburgh} EH9 3JE, U\lowercase{nited} K\lowercase{ingdom} \\
E\lowercase{mail}: T.L.C\lowercase{andy}@\lowercase{sms.ed.ac.uk}}}

\date{\today}
\dateposted{}

\thanks{AMS Subject Classification: 35Q41, 35A01.}
\begin{document}

\begin{abstract}
We prove global existence from $L^2$ initial data for a nonlinear Dirac equation known as the Thirring model \cite{Thirring1958}. Local existence in $H^s$ for $s>0$, and global existence for $s>\frac{1}{2}$, has recently been proven by Selberg and Tesfahun in \cite{Selberg2010b} where they used $X^{s, b}$ spaces together with a type of null form estimate. In contrast, motivated by the recent work of Machihara, Nakanishi, and Tsugawa, \cite{Machihara2010}  we first prove local existence in $L^2$ by using null coordinates, where the time of existence depends on the profile of the initial data. To extend this to a global existence result we need to rule out concentration of $L^2$ norm, or charge, at a point. This is done by decomposing the solution into an approximately linear component and a component with improved integrability. We then prove global existence for all $s\g0$.

\end{abstract}

\maketitle

\section{Introduction}
We consider the nonlinear Dirac equation
        \begin{equation}\label{Dirac1} \begin{split} -i \gamma^\mu \p_\mu u + m u&= \lambda (\overline{u} \gamma^\mu u ) \gamma_\mu u \\
                                                                         u(0)&=u_0\end{split}\end{equation}
where $u$ is a $\CC^2$ valued function of $(t,x)\in\RR^{1+1}$, and $m, \lambda \in \RR$. Indices are raised and lowered with respect to the Minkowski metric diag$(-1, 1)$, and repeated indices are summed over $\mu=0, 1$. The Dirac matrices $\gamma^{\mu}$ are defined by
        $$ \gamma^0 = \begin{pmatrix} 0 & 1 \\ 1 & 0 \end{pmatrix}, \qquad \gamma^1 = \begin{pmatrix} 0 & -1 \\ 1 & 0 \end{pmatrix} $$
and for a vector valued function $u$ we let $\overline{u}= u^{\dagger} \gamma^0$, where $u^{\dagger}$ denotes the conjugate transpose. The nonlinear Dirac equation (\ref{Dirac1}) is also known as the Thirring model and describes the vector self interaction of a Dirac field, see \cite{Thirring1958}. Classical solutions to $(\ref{Dirac1})$ satisfy conservation of charge
            $$ \| u(t) \|_{L^2_x} = \| u_0\|_{L^2_x}.$$
The scale invariant space is the charge class $L^2$, thus the equation is $L^2$ critical and so we expect the global well-posedness result proved below is sharp. However, we have no explicit counterexample to well-posedness for $s<0$.\\

Let $u= \begin{pmatrix} \psi \\ \phi \end{pmatrix}$ and $u_0= \begin{pmatrix} f \\ g \end{pmatrix}$. Writing out the equation (\ref{Dirac1}) in terms of $\phi$ and $\psi$ we obtain the system
\begin{equation}\label{Dirac}
        \begin{split}  \p_t \psi + \p_x \psi &= -i m \phi - i 2\lambda |\phi|^2 \psi \\
                        \p_t \phi - \p_x \phi &= -im \psi - i 2\lambda |\psi|^2 \phi\\
                        \big(\psi, \phi \big)^T(0) &= (f, g)^T \end{split} \end{equation}
where we take $f, g \in H^s$. In the classical case, $s\g1$, global existence was first proved by Delgado in \cite{Delgado1978} where he noticed that if $(\psi, \phi)$ is a  solution to (\ref{Dirac}), then $(|\psi|^2, |\phi|^2)$ satisfies a quadratic nonlinear Dirac equation. Thus, particularly for global in time problems, the nonlinearity is milder for the square of the solution. Together with Gronwall's inequality, Delgado used this  quadratic nonlinear Dirac equation to obtain an a priori bound on the $L^{\infty}$ norm of the solution. Since the time of existence can be shown to depend only on the $L^\infty$ norm, an application of the Sobolev embedding theorem shows the solution exists globally in time.

More recently, Selberg and Tesfahun used the $X^{s, b}$ spaces together with the null form type estimate\footnote{The term null form estimate is used for (\ref{null}) as the inequality relies crucially on the structure of the nonlinear term. In particular if we replace $|\phi|^2 \psi$ with $|\psi|^2 \psi$ then this estimate fails. The observation that null structure is needed to prove low regularity existence for nonlinear wave equations is due to Klainerman and Machedon in the seminal paper \cite{Klainerman1993}.}
					\begin{equation}\label{null} \big\| |\phi|^2 \psi \big\|_{X^{s, b-1 +\epsilon}_{+}} \lesa \| \phi\|_{X^{s, b}_-}^2 \|\psi \|_{X^{s, b}_+}\end{equation}
to prove local existence in the almost critical case $s>0$, where $X^{s, b}_+$ and $X^{s, b}_-$ are the $X^{s, b}$ spaces adapted to the linear propagators in (\ref{Dirac}) see \cite{Selberg2010b}. This estimate fails at the endpoint\footnote{  This can be seen by letting $\psi$ and $\phi$ be the relevant homogeneous solutions.} $s=0$ and so the approach using standard $X^{s, b}$ spaces seems limited to the case $s>0$. We also mention that the paper \cite{Selberg2010b} included global existence for $s>\frac{1}{2}$ by using the method of Delgado referred to above.

 Other  nonlinear Dirac equations in one dimension have also been studied. In \cite{Dias1986} and \cite{Machihara2007a} the closely related nonlinearity $|u|^2 u$ was considered. Local well-posedness results for quadratic nonlinearities have appeared in \cite{Bournaveas2008}, \cite{Machihara2007}, \cite{Machihara2005}, and \cite{Selberg2010b}. \\
\newline

In the current article we use null coordinates to prove global existence in $H^s$ for all $s \g0$, similar to the method used in the recent work of Machihara, Nakanishi, and Tsugawa \cite{Machihara2010}. The use of null coordinates has certain advantages over using the $X^{s, b}$ framework as we can work exclusively in the spatial domain and make use of the embedding $W^{1,1} \subset L^\infty$. Furthermore the local existence component of the proof is surprisingly straightforward. Once we change into null coordinates we will be forced to localise in both space and time. In the $L^2$ case this is not an issue as the Dirac equation satisfies finite speed of propagation. However, when trying to extend the global existence result to $s>0$, localising in both space and time will prove to be a little inconvenient and some technical results on localised Sobolev spaces will be required.

The time of existence of the local solution obtained below depends on the profile of the initial data. As a consequence, the conservation of charge property does not imply global existence. This is to be expected as we are dealing with an equation at a scale invariant regularity, see for instance \cite{Tao2006b} for a discussion related to the problem of proving global existence for the energy critical wave equation.  Thus, to obtain a global in time result, we need to have some control over the profile of the solution. This is done by modifying the approach of Delgado. Note that in previous works, the method of Delgado gave $L^\infty$ control of the solution provided the initial data belonged to $L^\infty$. Here however, we are working with low regularity solutions  and have no $L^\infty$ control over the initial data. So a new idea is required.

The way forward is to decompose our solution into two components. We show that the first of these components satisfies an essentially linear equation, while the second component has additional integrability, see Proposition \ref{decomp}. We remark that, since the Dirac equation in one dimension is roughly a coupled transport equation, the solution does not disperse. So generically we should not expect the solution to have any better integrability than the initial data. Thus the fact that we can decompose our solution into a linear piece and an $L^\infty$ piece is quite remarkable.
\newline

We now state our main result.

\begin{theorem}\label{global}
Let $s \g 0$ and $f, g \in H^s$. There exists a global solution $(\psi, \phi) \in C(\RR, H^s)$ to (\ref{Dirac}) such that the charge is conserved, so
            $$  \|\psi(t) \|_{L^2_x}^2 + \| \phi(t)\|_{L^2_x}^2  = \| f\|^2_{L^2_x} + \| g\|_{L^2_x}^2$$
for every $t \in \RR$. Moreover, the solution is unique in a subspace of $C(\RR, L^2_{loc})$ and we have continuous dependence on initial data.

\end{theorem}

The first step in the proof of Theorem \ref{global} is the following local in time result.

\begin{theorem}\label{local}
Let $f, g \in L^2$. There exists $T>0 $ such that we have  a solution $(\psi, \phi) \in C\big([-T, T], L^2\big)$ to (\ref{Dirac}). Moreover, the solution is unique in a subspace of $C\big([-T, T], L^2_{loc}\big)$ and we have continuous dependence on initial data.
\end{theorem}

In Theorem \ref{local} we require $T>0$ to satisfy,  for every $x \in \RR$,
				$$ \int_{|x-y|<2T} |f|^2 + |g|^2 dy < \epsilon$$
for a small $\epsilon>0$. Thus, as remarked above,  conservation of charge does not immediately lead to global existence.
\newline

We now give a brief outline of this article. In Section \ref{preliminaries} we introduce the function spaces we iterate in, as well the estimates we need for the proof of Theorem \ref{local}. Section \ref{local proof} contains the proof of Theorem \ref{local} and in Section \ref{global proof} we  prove Theorem \ref{global} in the case $f, g \in L^2$. Section \ref{section persist} extends the global result to $s>0$. Finally in the Appendix we have collected the proofs of some results on local Sobolev spaces which we require in Section \ref{section persist}.

\subsection*{Acknowledgements} The author thanks his supervisor Nikolaos Bournaveas, and also Shuji Machihara, for their encouragement and for many helpful discussions and suggestions. Thanks must also go to Selberg and Tesfahun, and  Machihara, Nakanishi, and Tsugawa, for making available preprints of their respective papers \cite{Selberg2010b} and \cite{Machihara2010}. These results will form part of the author's PhD thesis.

\subsection{Notation}
Throughout this paper $C$ denotes a positive constant which can vary from line to line. The notation $a \lesa b $ denotes the inequality $a \les C b$. If we wish to make explicit the fact that $C$ depends on a quantity $\Gamma$, we use a subscript $a \lesa_\Gamma b$.
 For a set $I \subset \RR$ and $1\les p \les \infty$ we let $L^p(I)$ denote the usual Lebesgue space.  We also use  the mixed version $L^p_t L^q_x(I_1 \times I_2)$ with norm
            $$ \|  u \|_{L^p_t L^q_x(I_1 \times I_2)} = \bigg( \int_{I_1} \bigg( \int_{I_2} |u(t, x)|^q dx \bigg)^{\frac{p}{q}} dt \bigg)^{\frac{1}{p}}$$
 where $I_1, I_2 \subset \RR$ and we make the obvious modification if $p, q = \infty$. Occasionally we write $L^p(\RR) = L^p$ when we can do so without causing confusion. This comment also applies to the other function spaces which appear throughout this paper. We let $C^{\infty}_0$ denote the space of smooth functions with compact support. If $X$ is a metric space and $I\subset \RR $ is an interval,  then  $C(I, X)$ denotes the set of continuous functions from $I$ into $X$. For $1<p<\infty$ and $s \in \RR$, we take $W^{s, p}(\RR)$ to be the usual Sobolev space defined using the norm
            $$ \| f \|_{W^{s, p} (\RR)} = \| \Lambda^s f \|_{L^p(\RR)}$$
 where $\widehat{\big(\Lambda^s f \big)} (\xi) = (1+ |\xi|^2)^{\frac{s}{2}} \widehat{f}(\xi)$ and $\widehat{f}$ denotes the Fourier transform of $f$. We also define the homogeneous variant $\dot{W}^{s, p}$ via the norm
            $$ \| f\|_{\dot{W}^{s, p}}  = \| D^s f \|_{L^p(\RR)}$$
 where $\widehat{D^s f}(\xi) = |\xi|^s \widehat{f}(\xi)$. In the special case $p=2$, we use the notation $H^s = W^{s, 2}$ and $\dot{H}^s = \dot{W}^{s, p}$. For an interval $I \subset \RR$ we define $W^{s, p}(I)$ by restricting elements of $W^{s, p}(\RR)$ to $I$. The restriction space $W^{s, p}(I)$ is a Banach space with norm
            $$ \| f \|_{W^{s, p}(I) }= \inf_{ g|_{I} = f} \| g\|_{W^{s, p}(\RR)}.$$
The homogeneous version $\dot{W}^{s, p}(I)$ is defined similarly.  Finally we let $I_R= [-R, R]$ and $\Omega_R = [-R, R] \times [-R, R]$ where $R>0$.\\

\section{Preliminaries}\label{preliminaries}

\begin{figure}
\setlength{\unitlength}{0.150mm}
\begin{picture}(522,435)(200,-480)
             \allinethickness{0.254mm}\special{sh 0.2}\path(260,-285)(440,-105)(620,-285)(440,-465)(260,-285) 
        \allinethickness{0.254mm}\put(440,-480){\vector(0,1){435}} 
        \allinethickness{0.254mm}\put(200,-285){\vector(1,0){520}} 
        \allinethickness{0.254mm}\path(440,-285)(620,-105)\special{sh 1}\path(620,-105)(615,-108)(616,-109)(617,-110)(620,-105) 
        \allinethickness{0.254mm}\path(440,-285)(280,-445) 
        \allinethickness{0.254mm}\path(440,-285)(620,-465)\special{sh 1}\path(620,-465)(615,-462)(616,-461)(617,-460)(620,-465) 
        \allinethickness{0.254mm}\path(440,-285)(265,-110) 
        \put(705,-310){\shortstack{$x$}} 
        \put(448,-61){\shortstack{$t$}} 
        \put(623,-456){\shortstack{$\beta= x-t$}} 
        \put(615,-126){\shortstack{$\alpha= x+t$}} 
        \put(460,-186){\shortstack{$\Omega_R$}} 
        \put(615,-310){\shortstack{$R$}} 
        \put(225,-310){\shortstack{$-R$}} 
\end{picture}
\caption{Null coordinates and the domain $\Omega_R$. }
\label{null coords pic}
\end{figure}
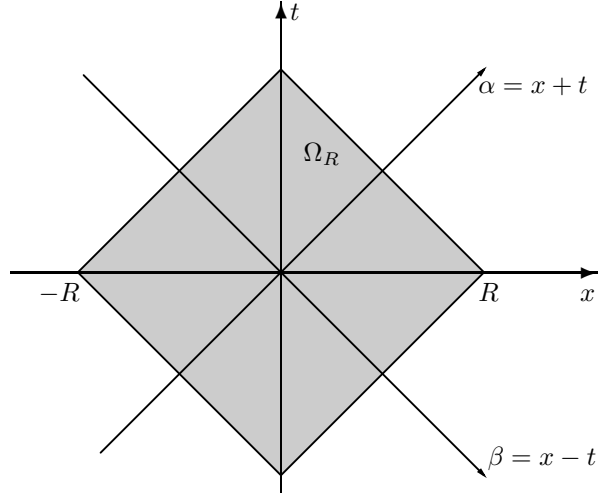

Define the spaces $Y_R$ and $X_R$ as the completion of $C^{\infty}$ using the norms
        $$ \| \psi \|_{Y_R} = \| \psi^* \|_{L^{\infty}_\alpha L^2_\beta(\Omega_R)} + \| \p_\alpha \psi^* \|_{L^1_\alpha L^2_\beta(\Omega_R)}$$
and
        $$ \| \phi \|_{X_R} = \| \phi^* \|_{L^{\infty}_\beta L^2_\alpha(\Omega_R)} + \| \p_\beta \phi^* \|_{L^1_\beta L^2_\alpha(\Omega_R)}$$
where $\psi(t, x)$ is a space time map into $\CC$ and we define $\psi^*(\alpha, \beta) = \psi\big(\frac{\alpha - \beta}{2} , \frac{\alpha + \beta}{2} \big)$. We refer to the coordinates $(\alpha, \beta)=(x+t, x-t)$ as null coordinates, see Figure \ref{null coords pic}. Note that if we compare the $Y_R$ and $X^{s, b}$ norms, it is easy to see what we gain by using null coordinates. Roughly speaking, for $X^{s, b}$ spaces we have $\frac{1}{2}$ a derivative in $L^2$ in the null direction, which fails to control $L^\infty$. In our norms however, we have a full derivative in $L^1$, which does control $L^\infty$, despite the fact that the norms $\| \cdot \|_{\dot{H}^{\frac{1}{2}}} $ and $\|\cdot\|_{\dot{W}^{1, 1}}$ have the same scaling.  \\

The $Y_R$ and $X_R$ norms are similar to those used in \cite{Machihara2010}, where they used norms of the form $\| \cdot\|_{L^2_\alpha L^\infty_\beta}$. In fact the norms $\| \cdot\|_{L^2_\alpha L^\infty_\beta}$ would suffice to give the $L^2$ case of Theorem \ref{global}. However using $L^2_\alpha L^\infty_\beta$ type spaces gives no control over derivatives in the null directions, which is required in the persistence of regularity argument in Section 5. Thus we need to use the slightly stronger $Y_R$, $X_R$ norms.\\

The first result we will need is the following energy type inequality.
\begin{proposition}\label{energy}
Assume $\psi$ is a solution to $ \p_t \psi + \p_x \psi = F$ with $\psi(0)= f$ and $f, F \in C^\infty$. Then
        $$ \| \psi \|_{Y_R} \les \|f\|_{L^2(I_{R})} + \| F^* \|_{L^1_\alpha L^2_\beta(\Omega_R)} .$$

Similarly, if $\phi$ solves $\p_t \phi - \p_x \phi =G$ with $\psi(0) = g$ and $g, G \in C^\infty$, then
        $$ \| \phi \|_{X_R} \les \|g\|_{L^2(I_{R})} + \| G^* \|_{L^1_\beta L^2_\alpha(\Omega_R)} .$$
\begin{proof}
We only prove the first inequality as the second is almost identical. Write the solution $\psi$ as
       $$ \psi(t, x) = f(x-t) + \int_0^t F(s , x-t+s) ds .$$
Then a simple change of variables gives
        $$\psi^*(\alpha, \beta) = f(\beta) + \frac{1}{2} \int_{\beta}^\alpha F^*(s, \beta) ds.$$
Therefore the proposition follows from the definition of $\| \cdot \|_{Y_R}$ together with Minkowski's inequality.

\end{proof}
\end{proposition}

The energy type inequality gains a full derivative in the relevant null direction, this gain of regularity will prove crucial and is a substitute for the null form estimates of the form (\ref{null}) used in \cite{Selberg2010b}.

We will also need the following  estimate which is essentially just the embedding $W^{1, 1} \subset L^\infty$.
\begin{lemma}\label{est}
For any $R>0$ we have
        $$ \|\psi^* \|_{L^2_\beta L^\infty_\alpha(\Omega_R)} \les \|\psi\|_{Y_R}$$
and
         $$ \|\phi^* \|_{L^2_\alpha L^\infty_\beta(\Omega_R)} \les \|\phi\|_{X_R}.$$
\begin{proof}
Since $C^\infty$ is dense in $Y_R$ and $X_R$, it suffices to consider the case $\psi, \phi \in C^\infty$. Then for every $\alpha, \beta \in \Omega_R$
        $$ \psi^*(\alpha, \beta) = \int_0^\alpha \p_\alpha \psi^*(\gamma, \beta)d\gamma + \psi^*(0, \beta).$$
Taking the supremum over $\alpha$ followed by the $L^2$ norm in $\beta$ gives the inequality for $\psi^*$. The inequality for $\phi^*$ is similar.

\end{proof}
\end{lemma}

\begin{corollary}\label{emb}
Let $0<T< R$. Then we have the continuous embeddings $Y_R, X_R \subset C\big(I_T, L^2(I_{R-T})\big)$.
\begin{proof}
Write $\psi(t, x) = \psi^*(x+t, x-t)$. Since $(t,x) \in [-T, T] \times [-R+T, R-T]$ and $0<T<R$ we have $|t+x|\les R$ and $|x-t|\les R$. Therefore
        \begin{align*}
            \| \psi \|_{L^\infty_t L^2_x(I_T \times I_{R-T}) }&= \| \psi^*(x+t, x-t)\|_{L^\infty_t L^2_x(I_T \times I_{R-T})} \\
                                &\les \| \psi^*(\alpha, \beta) \|_{L^2_\beta L^\infty_\alpha(\Omega_R)}\end{align*}
and so the previous lemma gives
        \begin{equation}\label{emb eqn1} \| \psi \|_{L^\infty_t L^2_x(I_T \times I_{R-T}) } \les \| \psi \|_{Y_R}. \end{equation}
The $L^2$ continuity of $\psi(t)$ then follows from the uniform bound  (\ref{emb eqn1}) together with the density of $C^\infty$  in $Y_R$. The embedding $X_R \subset C\big( I_T , L^2(I_{R-T}) \big)$ follows from a similar application of Lemma \ref{est}.
\end{proof}
\end{corollary}

\section{Local Existence}\label{local proof}

We will deduce Theorem \ref{local} from the following localised version via translation invariance.

\begin{theorem}\label{localised}
Let $0<R<\frac{1}{16|m|}$. There exists $\epsilon>0$ depending only on $\lambda$ such that if $ f, g \in L^2(I_R)$ satisfy
		\begin{equation}\label{localised cond} \|f\|_{L^2(I_R)} + \|g\|_{L^2(I_R)} < \epsilon,\end{equation}
then there exists a unique solution $(\psi, \phi) \in Y_R \times X_R$ to (\ref{Dirac}) such that
        $$ \| \psi\|_{Y_R} + \| \phi \|_{X_R} < 2 \epsilon.$$
Moreover the solution map, mapping initial data satisfying the condition (\ref{localised cond}) to the solution $(\psi, \phi) \in Y_R \times X_R$, is Lipschitz continuous.

\begin{proof}

Let $$\mathcal{X}_R = \big\{ (\psi, \phi) \in Y_R \times X_R \big| \,\, \| \psi\|_{Y_R} + \|\phi\|_{X_R} < 2\epsilon\,\}$$ where  $\epsilon>0$ is a small constant to be fixed later. Define $N_R:\mathcal{X}_R \rightarrow \mathcal{X}_R$ by $N_R(\psi, \phi) = (u, v)$ where

    \begin{align*}
        u(t, x) &= f(x-t) - \int_0^t \big(i m \phi + i 2\lambda |\phi|^2\psi\big)(s, x-t+s) ds\\
        v(t, x) &= g(x+t) - \int_0^t \big( i m \psi + i 2\lambda |\psi|^2 \phi \big) (s, x+t - s) ds. \end{align*}

By Proposition \ref{energy} we have
        $$ \| u \|_{Y_R} + \|v\|_{X_R} \les  \epsilon +  \| m \phi^* \|_{L^1_\alpha L^2_\beta(\Omega_R)}  + \| m \psi^* \|_{L^1_\beta L^2_\alpha(\Omega_R)}  +  2 \big\|  \lambda |\phi^*|^2 \psi^* \big\|_{L^1_\alpha L^2_\beta(\Omega_R)}+  2 \big\|  \lambda |\psi^*|^2\phi^*\big\|_{L^1_\beta L^2_\alpha(\Omega_R)} .$$
An application of H\"{o}lder's inequality shows that
            \begin{align*} \| \phi^* \|_{L^1_\alpha L^2_\beta(\Omega_R)} + \| \psi^* \|_{L^1_\beta L^2_\alpha(\Omega_R)}
                 &\les 2R \| \phi^*\|_{L^\infty_\beta L^2_\alpha(\Omega_R)} + 2R \| \psi^* \|_{L^\infty_\alpha L^2_\beta(\Omega_R)}\\
                 &\les 4R \epsilon. \end{align*}
The nonlinear terms can be controlled by H\"{o}lder's inequality followed by Lemma \ref{est}, for instance
        \begin{align*} \big\|  |\phi^*|^2 \psi^* \big\|_{L^1_\alpha L^2_\beta(\Omega_R)}
                    &\les \| \phi^* \|_{L^2_\alpha L^\infty_\beta(\Omega_R)}^2 \|\psi^*\|_{L^\infty_\alpha L^2_\beta(\Omega_R)} \\
                 &\les \|\phi\|_{X_R}^2 \|\psi\|_{Y_R} \end{align*}
the remaining term is similar. Combining these estimates we obtain
    $$ \| u \|_{Y_R} + \|v\|_{X_R} \les \epsilon + \big( 4 R |m| + 16 |\lambda|  \epsilon^2 \big) \epsilon.$$
Therefore provided $0<R<\frac{1}{16  |m|}$ and $\epsilon$ is sufficiently small (depending only on $\lambda$), we see that $N_R$ is well defined. To show $N_R$ is a contraction mapping follows by a similar application of Proposition \ref{energy}, thus we obtain existence. Continuous dependence on initial data in $\mathcal{X}_R$ is a simple corollary of the estimates used to deduce that $N_R$ is a contraction mapping.

It only remains to prove uniqueness.  Assume we have a solution $(\psi', \phi') \in Y_R \times X_R$ with initial data $(f, g)\in L^2(I_R)$ satisfying $(\ref{localised cond})$ and let $(\psi, \phi)$ denote the solution constructed by the above fixed point argument with the same initial data $(f, g)$. By choosing $R'\les R$ sufficiently small we have
                $$ \| \psi' \|_{Y_{R'}} + \| \phi' \|_{X_{R'}} < 2 \epsilon$$
and so $(\psi', \phi') \in \mathcal{X}_{R'}$. Note that we also have $(\psi, \phi) \in \mathcal{X}_{R} \subset \mathcal{X}_{R'}$. Thus, as there is a unique fixed point in $\mathcal{X}_{R'}$, we deduce that
            $ (\psi', \phi') = (\psi, \phi)$
on $\Omega_{R'}$. Define
            $$R_{max} = \sup \big\{ r \les R\, \big| \, \| \psi' \|_{Y_{r}} + \| \phi' \|_{X_{r}} < 2 \epsilon \big\}$$
and suppose $R_{max}<R$. Then by the above argument we have $ (\psi', \phi') = (\psi, \phi)$ on $\Omega_r$ for every $r<R_{max}$ and hence
           \begin{align*} \| \psi' \|_{Y_{R_{max}}} + \| \phi' \|_{X_{R_{max}}} &= \| \psi\|_{Y_{R_{max}}} + \|\phi\|_{X_{R_{max}}} \\
                                    &\les \| \psi\|_{Y_R} + \|\phi\|_{X_R}  < 2\epsilon. \end{align*}
Consequently $(\psi', \phi') \in \mathcal{X}_r$ for some $r>R_{max}$, contradicting the definition of $R_{max}$. Therefore we must have $R_{max}=R$ and so our solutions agree on $\Omega_R$. Finally, we note that by uniqueness, the continuous dependence on initial data extends from $\mathcal{X}_R$ to $Y_R \times X_R$.

\end{proof}
\end{theorem}

We can now prove Theorem \ref{local} by using translation invariance and uniqueness.

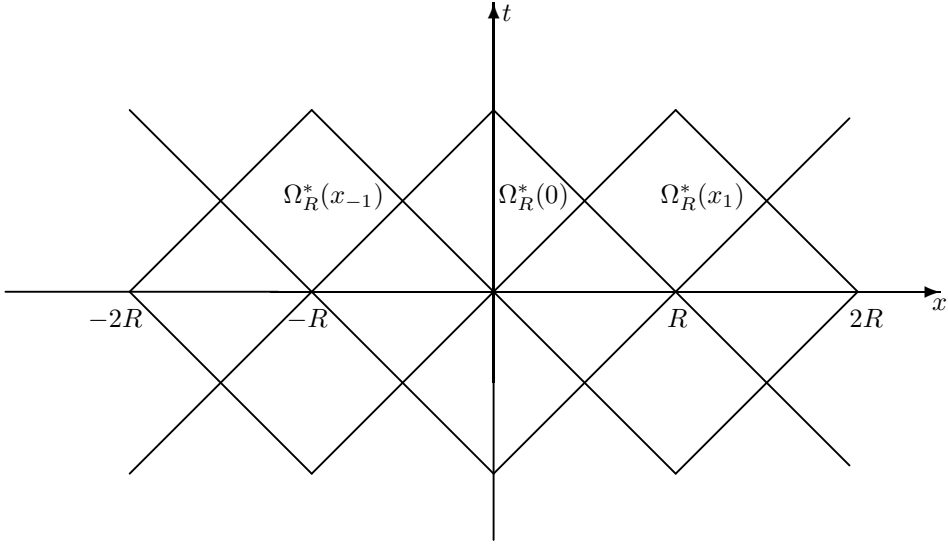
\begin{figure}
\setlength{\unitlength}{0.220mm}
\begin{picture}(577,326)(100,-415)
        \allinethickness{0.254mm}\put(260,-265){\vector(1,0){405}} 
        \allinethickness{0.254mm}\put(395,-320){\vector(0,1){230}} 
        \allinethickness{0.254mm}\path(395,-155)(505,-265) 
        \allinethickness{0.254mm}\path(505,-265)(395,-375) 
        \allinethickness{0.254mm}\path(395,-375)(285,-265) 
        \allinethickness{0.254mm}\path(285,-265)(395,-155) 
        \allinethickness{0.254mm}\path(395,-265)(505,-155) 
        \allinethickness{0.254mm}\path(505,-155)(615,-265) 
        \allinethickness{0.254mm}\path(615,-265)(505,-375) 
        \allinethickness{0.254mm}\path(395,-265)(505,-375) 
        \allinethickness{0.254mm}\path(505,-265)(610,-160) 
        \allinethickness{0.254mm}\path(505,-265)(610,-370) 
        \allinethickness{0.254mm}\path(395,-265)(285,-155) 
        \allinethickness{0.254mm}\path(285,-155)(175,-265) 
        \allinethickness{0.254mm}\path(175,-265)(285,-375) 
        \allinethickness{0.254mm}\path(395,-265)(285,-375) 
        \allinethickness{0.254mm}\path(285,-265)(175,-155) 
        \allinethickness{0.254mm}\path(285,-265)(175,-375) 
        \allinethickness{0.254mm}\path(395,-285)(395,-415) 
        \allinethickness{0.254mm}\path(265,-265)(100,-265) 
        \put(660,-276){\shortstack{$x$}} 
        \put(400,-101){\shortstack{$t$}} 
        \put(398,-211){\shortstack{$\Omega^*_R(0)$}} 
        \put(496,-211){\shortstack{$\Omega^*_R(x_1)$}} 
        \put(268,-211){\shortstack{$\Omega^*_R(x_{-1})$}} 
        \put(500,-286){\shortstack{$R$}} 
        \put(270,-286){\shortstack{$-R$}} 
        \put(610,-286){\shortstack{$2R$}} 
        \put(150,-286){\shortstack{$-2R$}} 
\end{picture}
\caption{The regions $\Omega^*_R(x_j)$ in the proof of Theorem \ref{local}.}
\label{proof of local exist pic}
\end{figure}
\begin{proof}[Proof of Theorem \ref{local}]

 Assume $f, g \in L^2$ and let $I_R(x) = [x-R, x+R]$. Choose $R$ sufficiently small so that
            \begin{equation}\label{local eqn1}
                    \sup_{j \in \ZZ}\Big(  \| f\|_{L^2(I_R(jR)) } + \| g\|_{L^2(I_R(jR))}\Big) < \epsilon
            \end{equation}
where $0<R<\frac{1}{16|m|}$ and $\epsilon>0$ is the constant in Theorem \ref{localised}. By Theorem \ref{localised} and spatial invariance we then get a solution $(\psi_j, \phi_j) \in Y_{R, x_j} \times X_{R, x_j}$, where $Y_{R, x_j}$ denotes the $Y_R$ space centered at $x_j=j R$ with radius $R$. Using uniqueness we can glue these solutions together to get a solution $(\psi, \phi)$ on $\bigcup_{j \in \ZZ} \Omega_{R}^*(x_j)$ where
				$$ \Omega^*_R(x) = \big\{ (t, y) \, \big| \, |t + y-x|\les R, \,\, |t-y+x|\les R \,\big\}.$$
Letting $T=\frac{R}{2}$ and noting that $|x_j - x_{j+1}|=R$ we have $[-T, T] \times \RR \subset \bigcup_{j \in \ZZ} \Omega_{R}^*(x_j)$. Thus it only remains to prove that, firstly, $(\psi, \phi) \in C([-T, T], L^2)$ and secondly, that the solution map is continuous.

To this end assume $(f_k, g_k)$ converges to $(f, g)$ in $L^2$. By choosing $N>0$ sufficiently large we can ensure that $(f_k, g_k)$ satisfies (\ref{local eqn1}) for every $k\g N$ and $j\in \ZZ$. Using Theorem \ref{localised} and repeating the above argument we then get a solution $(\psi_k, \phi_k)$ on $[-T, T]\times \RR$. Moreover, the Lipschitz continuity of the localised solution map together with the embedding of Corollary \ref{emb} gives for every $ |t|<T$
            $$ \int_{|x-x_j|\les T } |\psi(t) - \psi_k(t)|^2 + |\phi(t) - \phi_k(t)|^2 dx \lesa \int_{|x-x_j| \les 2T} |f - f_k|^2 + |g - g_k|^2 dx. $$
Summing these inequalities over $j \in \ZZ$ we obtain
                $$ \| \psi - \psi_k\|_{L^\infty_t L^2_x} + \| \phi - \phi_k \|_{L^\infty_t L^2_x} \lesa \| f -f_k \|_{L^2} + \| g - g_k \|_{L^2}$$
and so the solution map is continuous. It is also now easy to see that $(\psi, \phi) \in C([-T, T], L^2)$.

\end{proof}

\section{Global Existence}\label{global proof}

We start by showing global existence forward in time, existence backwards in time will then follow by a symmetry argument. Suppose we tried to iterate forwards the local in time result of Theorem \ref{local}. Then we would obtain a sequence of strictly increasings times $T_0<T_1<...$ and  a solution on $[0, T_j]$, where the size of each $T_j$ would depend only on how small we needed to make $R$ before
				$$ \sup_{y\in\RR} \int_{|x-y|<R} |\psi(T_{j-1}, x)|^2 + |\phi(T_{j-1}, x)|^2 dx < \epsilon.$$
Thus, roughly speaking, provided we can ensure $R$ does not shrink to zero, we would obtain global existence. Note that the usual conservation of charge property is not sufficient, as it does not prevent the charge from concentrating at  a point. Instead we need to make use of the structure of the equation (\ref{Dirac}) via an argument similar to that of Delgado \cite{Delgado1978}.

\begin{proposition}\label{decomp}
Let $T>0$ and $2\les p \les \infty$. Assume $(\psi, \phi)\in C^\infty$ is a solution to (\ref{Dirac}) with initial data $f, g \in C^\infty_0$. Then there exists a decomposition
            $$(\psi, \phi) = (\psi_L, \phi_L)  + ( \psi_N, \phi_N)$$
such that
            $$ |\psi_L(t, x)| = |f(x-t)|, \qquad |\phi_L(t, x) | = |g(x+t)|,$$
and for every $0\les t \les T$,
            $$ \|\psi_N(t)\|_{L^p_x} + \|\phi_N(t)\|_{L^p_x} \lesa_{m, T} \| f\|_{L^2} +  \|g\|_{L^2}.$$
\begin{proof}
Assume $(f, g) \in C^\infty_0$ and let $\psi$, $\phi$ denote the corresponding (smooth) solutions to $(\ref{Dirac})$. Let $(\psi_N, \phi_N)$ be the solution to
                \begin{align*}
                        \p_t \psi_N + \p_x \psi_N  &= -i m \phi - i2 \lambda |\phi|^2 \psi_N \\
                         \p_t \phi_N - \p_x \phi_N  &= -i m \psi - i2 \lambda |\psi|^2 \phi_N
                \end{align*}
with $\psi_N(0) = \phi_N(0) = 0$ and let $(\psi_L, \phi_L)$ be the solution to
\begin{align*}
                        \p_t \psi_L + \p_x \psi_L  &=  - i 2\lambda |\phi|^2 \psi_L \\
                         \p_t \phi_L - \p_x \phi_L  &= - i 2\lambda |\psi|^2 \phi_L
                \end{align*}
with initial data $\psi_L(0) = f$ and $\phi_L(0) = g$. Note that by uniqueness of smooth solutions we have $(\psi, \phi)  = (\psi_L, \phi_L)  + (\psi_N, \phi_N)$. A computation shows that
                \begin{align*} \p_t |\psi_N|^2 + \p_x |\psi_N|^2 &= 2 m \Im( \phi \overline{\psi}_N) \\
                \p_t |\phi_N|^2 - \p_x |\phi_N|^2 &= 2 m \Im( \psi \overline{\phi}_N)
                \end{align*}
and
                \begin{align*} \p_t |\psi_L|^2 + \p_x |\psi_L|^2 &= 0\\
                \p_t |\phi_L|^2 - \p_x |\phi_L|^2 &= 0
                \end{align*}
where $\Im(z)$ denotes the imaginary part of $z \in \CC$. Thus we can write the solutions $\psi_L$ and $\psi_N$ as $ |\psi_L| = |f(x-t)|$ and
						\begin{equation}\label{mass1}  |\psi_N(t, x)|^2 = 2 m \int_0^t \Im(\phi \overline{\psi}_N )(s, x-t+s) ds. \end{equation}
Since $\phi = \phi_L+ \phi_N$ and $|\phi_L(t, x)| = |g(x+t)|$ we have
                \begin{align*}
                    \Big| \int_0^t \Im(\overline{\phi} \psi_N )(s, x-t+s) ds \Big|&\lesa \int_0^t |\phi(s, x-t+s)|^2 ds + \int_0^t |\psi_N(s, x-t+s) |^2ds \\
                                                                        &\lesa \int_0^t |g(2s + x-t)|^2 ds + \int_0^t |\phi_N(s, x-t+s) |^2 + |\psi_N(s, x-t+s) |^2 ds \\
                                                                        &\lesa \| g \|_{L^2}^2 + \int_0^t \|\phi_N(s) \|_{L^\infty_x}^2+ \|\psi_N(s)\|_{L^\infty_x}^2 ds.\end{align*}
Taking the $L_x^{\infty}$ norm of both sides of (\ref{mass1}) we obtain
                $$ \| \psi_N(t) \|_{L^\infty_x}^2 \lesa_m \| g\|_{L^2}^2 + \int_0^t \| \psi_N(s) \|_{L^\infty_x}^2 + \| \phi_N(s) \|_{L^\infty_x}^2 ds. $$
A similar argument gives
		$$\|\phi_N(t) \|_{L^\infty_x}^2 \lesa_m  \|f\|_{L^2}^2  + \int_0^t \| \psi_N(s)\|^2_{L^\infty_x} + \| \phi_N(s) \|_{L^\infty_x}^2 ds.$$
Therefore using Gronwall's inequality we see that for every $0 \les t \les T$ we have 			
			$$\|\psi_N(t) \|_{L^\infty_x} + \| \phi_N(t) \|_{L^\infty_x} \lesa_{m, T}  \|f\|_{L^2} +  \|g\|_{L^2}. $$
The finiteness of the $L^2$ norm follows by using conservation of charge
        \begin{align*}
            \| \psi_N(t) \|_{L^2_x} + \| \phi_N(t) \|_{L^2_x} &\les \| \psi(t) \|_{L^2_x} + \| \phi(t)\|_{L^2_x} + \| \psi_L(t)\|_{L^2_x} + \| \phi_L(t)\|_{L^2_x} \\
            &\lesa \| f\|_{L^2_x} + \| g\|_{L^2_x} .\end{align*}
Thus result follows by interpolation.
\end{proof}
\end{proposition}

The above proposition contains the decomposition alluded to in the introduction. Essentially the term $\psi_L$ is linear while the remaining term, $\psi_N$, has vanishing initial data and more integrability than one would naively expect. This additional integrability will then allow us to rule out concentration of charge.\\

The proof of Theorem \ref{global} in the case $s=0$ is now straightforward.

\begin{proof}[Proof of Theorem \ref{global} in the case $s=0$]
Suppose  $f, g \in L^2$ and let $(\psi, \phi) \in C\big( [0, T) , L^2\big)$ be the corresponding solution to (\ref{Dirac}) where $[0, T)$ is the maximal forward time of existence.  By Theorem \ref{local} it suffices to prove that
        \begin{equation}\label{global eqn1} \limsup_{t \rightarrow T} \,\sup_{x\in \RR}  \int_{|x-y|<4( T-t )} |\psi(t)|^2 + |\phi(t)|^2 dy = 0. \end{equation}
Since assuming (\ref{global eqn1}) holds, there exists $0<t^*<T$ such that $4(T-t^*)<\frac{1}{8|m|}$ and
            $$\sup_{x\in \RR}  \int_{|x-y|<4(T-t^*)} |\psi(t^*)|^2 + |\phi(t^*)|^2 dy < \epsilon $$
where $\epsilon = \epsilon(\lambda)$ is the small constant from the proof of Theorem \ref{local}. Taking $(\psi(t^*), \phi(t^*))$ as new initial data, by Theorem \ref{local} we can extend the solution to $[0, T) \cup [t^*, t^* + 2(T-t^*)]$.
 However, since $t^* + 2(T - t^*) >T$, this contradicts the assumptions that $[0, T)$ was the maximal forward time of existence. Therefore we must have $T=\infty$ and so solution exists globally in time.

 We now prove (\ref{global eqn1}). Since the solution depends continuously on the initial data, we may assume that $f, g \in C^\infty_0$. An application of Proposition \ref{decomp} with $p=\infty$ shows that
        \begin{align*}
               \sup_{x \in \RR}  \int_{|y-x|< 4(T-t) } |\psi(t)|^2 + |\phi(t)|^2 dy &\lesa \sup_{x \in \RR} \int_{|y-x|< 4(T-t)} |f(y-t)|^2 + |g(y+t)|^2 dy \\
               &\qquad \qquad \qquad \qquad \qquad \qquad+ (T-t) \big( \| \psi_N(t) \|_{L^\infty_y}^2 + \| \phi_N(t) \|_{L^\infty_y}^2 \big) \\
                            &\lesa_T \sup_{x \in \RR} \int_{|y-x+T|<5( T-t )} |f(y)|^2 dy \\
                             &\qquad \qquad + \sup_{x \in \RR} \int_{|y - x - T| < 5(T-t)} |g(y)|^2 dy + (T-t) \big( \| f \|_{L^2}^2 + \| g\|_{L^2}^2 \big).
        \end{align*}
 Hence letting $t$ tend to $T$ we obtain (\ref{global eqn1}) and so we have global existence forward in time.

 To obtain global existence backwards in time suppose $(\psi, \phi)$ is a solution to (\ref{Dirac}) on $(-T, 0]$ and define  $\psi'(t, x) = \phi(-t, x)$ and $\phi'(t, x) = \psi(-t, x)$. Then $(\psi', \phi')$ solves (\ref{Dirac}) on $[0, T)$ with $m$ and $\lambda$ replaced with $-m$ and $-\lambda$. The forwards in time argument above then shows that we can extend $(\psi', \phi')$ to $[0, \infty)$. Undoing the time reversal we see that we have a solution $(\psi, \phi)$ on $(-\infty, 0]$. Therefore for every initial data $f, g \in L^2$ we have a global solution $(\psi, \phi) \in C(\RR, L^2)$ to (\ref{Dirac}).

\end{proof}

\section{Persistence of Regularity}\label{section persist}

In this section we extend the global result for $s> \frac{1}{2}$ of Selberg and Tesfahun \cite{Selberg2010b} to $s>0$. This completes the proof of Theorem \ref{global}.  Ideally, since we already have global existence for $s=0$, we would like to include the $s>0$ result in the $L^2$ iteration scheme by using the standard persistence of regularity type arguments. However, since the $Y_R$, $X_R$ norms contain derivatives in $L^1$, they do not interact very well with fractional derivatives. Consequently, the proof of global existence for $s>0$ is slightly more complicated than the $L^2$ case and some technical results on localised Sobolev spaces are required.  We remark that we still make no use of the $X^{s, b}$ type spaces, thus null coordinates can also be used for $s >0$, see also \cite{Machihara2010}.\\

 The main result we prove in this section is the following.

\begin{theorem}\label{local bound}
Let $0<s<\frac{1}{4}$ and $ \frac{1}{2} = \frac{1}{p} + s$. There exists a  small constant $0<\epsilon^*<1$  such that if $|m|<\epsilon^*$ and $f, g \in H^s$ satisfy
            \begin{equation}\label{local bound eqn1} \| f\|_{L^p} + \| g\|_{L^p} < \epsilon^*, \end{equation}
  then there exists a solution $(\psi, \phi) \in C([-1,1] , H^s)$ solving (\ref{Dirac}). Moreover, the solution is unique in a subspace of $C([-1,1], H^s)$ and depends continuously on the initial data.
\end{theorem}

The small mass assumption in Theorem \ref{local bound} is required as the interval of existence is $[-1, 1]$. To motivate this consider the local existence result in $L^2$, Theorem \ref{localised}, where we needed the time of existence\footnote{Note that the size of the domain $\Omega_R$ was essentially the time of existence of a solution.} to satisfy $T\lesa \frac{1}{|m|}$. Thus if $T=1$ we have to take the mass, $m$, to be small. \\

Assuming Theorem \ref{local bound} holds, the proof of Theorem \ref{global} is straightforward.

\begin{proof}[Proof of Theorem \ref{global} in the case $s>0$]
The persistence of regularity proven by Selberg and Tesfahun in \cite{Selberg2010b}, reduces the problem to showing global existence for $0<s<\frac{1}{4}$. We now make use of a simple scaling argument. Take $f, g \in H^s$ and define $f_\tau= \tau^{\frac{1}{2}} f( \tau x)$, $g_\tau = \tau^{\frac{1}{2}} g(\tau x)$, $m' = \tau m$. By choosing $\tau$ sufficiently small we see that $f, g,$ and $m'$ satisfy the conditions of Theorem \ref{local bound}. Therefore we get a solution $(\psi_\tau, \phi_\tau)  \in C([-1, 1], H^s)$ to (\ref{Dirac}) with $m$ replaced by $m'$. To undo the scaling we let $\psi(t, x) = \tau^{-\frac{1}{2}} \psi_\tau(\tau^{-1} t, \tau^{-1}x)$ and define $\phi$ similarly. It is easy to see that $(\psi, \phi)$ is a solution to (\ref{Dirac}) with $(\psi, \phi) \in C([-T, T], H^s)$ where $T$ only depends on the size of some negative power of $\| f\|_{L^p} + \| g \|_{L^p}$. To conclude the proof we note that the decomposition in Proposition \ref{decomp} shows that $\| \psi(t) \|_{L^p} + \| \phi(t) \|_{L^p}< C(t)<\infty$ for every $t \in \RR$ where $C(t) \in L^\infty_{loc}(\RR)$.  Therefore the solution must exist globally in time.
\end{proof}

We have reduced the proof of global existence for $s>0$ to proving the local result in Theorem \ref{local bound}. The main tool to do this is again the use of null coordinates together with a decomposition along the lines of Proposition \ref{decomp}.\\

We now present some results on localised Sobolev spaces which we require in the proof of Theorem \ref{local bound}. To start with note that any inequality for $W^{s, p}(\RR)$ implies a corresponding inequality for the localised space $W^{s, p}(I)$ for any interval $I \subset \RR$. In particular, if $\frac{1}{q} \les \frac{1}{p} + s$ and $1<p\les q <\infty$, we have Sobolev embedding
        $$ \| f\|_{L^p(I)} \lesa  \|f \|_{W^{s, q}(I)}$$
and if $0<s<\frac{1}{2}$ and $y \in \RR$ we  have Hardy's inequality\footnote{See for instance page 334 of \cite{Tao2006b} for a proof of Hardy's inequality on $\RR^n$.}
        $$ \bigg\| \frac{f(x)}{ |x-y|^s} \bigg\|_{L^2_x(I)} \lesa \| f\|_{\dot{H}^s(I)}.$$
We also make use of the following well known characterisation of localised Sobolev spaces.

\begin{theorem}\label{Intrins}
Let $0<s<\frac{1}{2}$.
\begin{enumerate}
    \item Then\footnote{We use $ a \backsimeq b$ to denote the set of inequalities $a \lesa b \lesa a$.}
                $$ \| f\|_{H^s(I_R)}^2 \backsimeq_R \|f \|_{L^2(I_R)}^2 + \int_{I_R} \int_{I_R} \frac{ | f(x) - f(y) |^2}{|x-y|^{1 + 2s} } dx dy$$

    \item  Take $I_R(x) = [x - R, x+R]$. Then
                     $$ \| f\|_{H^s(\RR)}^2 \lesa \sum_{j \in \ZZ} \|f\|_{H^s(I_1(j))}^2 $$
        and
                        $$ \sum_{j \in \ZZ} \|f\|_{H^s(I_{2} (j))}^2 \lesa \| f\|_{H^s(\RR)}^2.$$

    \item If $\frac{1}{2}= \frac{1}{p} + s$ and $s<\frac{1}{q} <1$ then we have
            $$ \big\| |g|^2 f \big\|_{H^s(I_2)} \lesa \| g \|_{L^\infty(I_2)}^2 \| f \|_{H^s(I_2)} + \| g\|_{L^\infty(I_2)} \| g\|_{W^{1, q}(I_2)} \| f\|_{L^p(I_2)} .$$

\end{enumerate}
\begin{proof}
See the appendix.
\end{proof}
\end{theorem}

We also require the following estimates.

\begin{proposition}\label{int est}
Let $0<s<\frac{1}{2} $.
\begin{enumerate}
 \item Then
            $$ \Big\| \int_{-2}^x F(t', x) dt' \Big\|_{H^s_x(I_2)} \lesa\| F\|_{L^1_t H^s_x(\Omega_2)}.$$

\item We have
            $$ \| \phi^* \|_{L^1_\alpha H^s_\beta ( \Omega_2) } \lesa  \| \phi \|_{X_2} $$
     and
            $$ \| \psi^* \|_{L^1_\beta  H^s_\alpha( \Omega_2) } \lesa  \| \psi \|_{Y_2}. $$
\end{enumerate}
\begin{proof}
We start with $(i)$. The characterisation in Theorem \ref{Intrins} shows that
            $$\Big\| \int_{-2}^x F(t', x) dt' \Big\|_{H^s_x(I_2)}^2 \lesa \Big\|  \int_{-2}^x F(t', x) dt' \Big\|_{L^2_x(I_2)}^2+
                    \int_{I_2} \int_{I_2} \frac{ \big| \int_{-2}^x F(t', x) dt'  - \int_{-2}^y F(t', y) dt' \big|^2}{|x-y|^{1+2s}} dydx. $$
The $L^2$ component is easily controlled by using Minkowski's inequality. For the remaining part we note that, by symmetry, we may assume $x>y$. Then using the inequality
        \begin{align*} \big| \int_{-2}^x F(t', x) dt'  - \int_{-2}^y F(t', y) dt'\big| &\les \int_y^x | F(t', x) | dt' + \int_{-2}^y  | F(t', x) - F(t', y) |dt' \\
                &\les \int_y^x | F(t', x) | dt' + \| F(x) - F(y)\|_{L^1_t(I_2)} \end{align*}
we reduce to estimating the integrals
        $$\int_{I_2} \int_{I_2} \frac{ \big(  \int_y^x |F(t', x)|  dt' \big)^2}{|x-y|^{1+2s}} dydx$$
and
        $$\int_{I_2} \int_{I_2} \frac{ \| F(x) - F(y)\|_{L^1_t(I_2)}^2 }{|x-y|^{1+2s}} dydx .$$
The latter is again easily controlled by an application of Minkowski's inequality and so it only remains to estimate the former. To this end note that \begin{align*}
       \int_{I_2} \int_{I_2} \frac{ \big(  \int_y^x |F(t', x)|  dt' \big)^2}{|x-y|^{1+2s}} dydx &\les \bigg( \int_{I_2} \Big( \int_t^2\int_{-2}^t \frac{ |F(t, x)|^2 }{|x-y|^{1+2s}} dy dx \Big)^{\frac{1}{2} }dt \bigg)^2\\
                            &\lesa \bigg( \int_{I_2} \Big( \int_t^2 \frac{ |F(t, x)|^2 }{|x-t|^{2s}} + \frac{ |F(t, x)|^2 }{|x+2|^{2s}}  dx \Big)^{\frac{1}{2}} dt \bigg)^2\\
                            &\les \bigg( \int_{I_2} \bigg\| \frac{F(t, x)}{|x+2|^s} \bigg\|_{L^2_x(I_2)} + \bigg\| \frac{F(t, x) }{|x-t|^s}\bigg \|_{L^2_x(I_2)} dt \bigg)^2 \\
                            &\lesa \| F\|^2_{L^1_t H^s_x (\Omega_2)}
       \end{align*}
where we needed $0<s<\frac{1}{2}$ to apply Hardy's inequality.

To prove the first inequality in $(ii)$ we note that by H\"{o}lder's inequality together with Lemma \ref{est} it suffices to show that for all $\alpha \in I_2$
            $$ \| \phi^* \|_{H^s_\beta(I_2)} \lesa \| \phi^* \|_{L^\infty_\beta(I_2)} + \| \p_\beta \phi^* \|_{L^1_\beta(I_2)}.$$
To this end we note that
        \begin{align*}
             \| \phi^* \|_{H^s_\beta(I_2)}^2  &\lesa \| \phi^* \|_{L^2_\beta(I_2)}^2 +  \int_{I_2} \int_{I_2} \frac{|\phi^*(\sigma) - \phi^*(\gamma) |^2}{|\sigma - \gamma |^{1+2s} } d \sigma d \gamma \\
                                        &\lesa \| \phi^* \|_{L^\infty_\beta(I_2)}^2 +  \| \phi^* \|_{L^\infty_\beta(I_2)} \int_{I_2} \int_{I_2} \int_\sigma^\gamma \frac{|\p_\beta \phi^*(\beta)| }{|\sigma - \gamma |^{1+2s} }d\beta d \sigma d \gamma
        \end{align*}
where, as before,  we may assume $\sigma < \gamma$. To control the integral term we just change the order of integration to obtain
        \begin{align*}
            \int_{I_2} \int_{I_2} \int_\sigma^\gamma \frac{|\p_\beta \phi^*(\beta) |}{|\sigma - \gamma |^{1+2s} }d\beta d \sigma d \gamma
                            &= \int_{I_2} |\p_\beta \phi^*(\beta)| \int_\beta^2 \int_{-2}^\beta | \gamma - \sigma |^{-1 - 2s} d \gamma d \sigma d \beta \\
                            &\lesa \| \p_\beta \phi^* \|_{L^1(I_2)}
        \end{align*}
since $0<s<\frac{1}{2}$. Therefore the first inequality in $(ii)$ follows. The second is similar and we omit the details.
\end{proof}
\end{proposition}
For the remainder of this paper we  fix $p$,  $q$ such that
    \begin{equation}\label{p defn} \frac{1}{2} = \frac{1}{p} + s \end{equation}
and
    \begin{equation}\label{q defn} \frac{1}{q} = 1 - 2s. \end{equation}
Note that $2q = p$ and for $s<\frac{1}{4}$, we have $2<p<4$ and $1<q<2 $. Also by Sobolev embedding,  $H^s(I) \subset L^p(I)$.  Define the spaces $Y^s$ and $X^s$ by using the norms
            $$ \| \psi \|_{Y^{s}} = \| \psi^* \|_{L^\infty_\alpha H^s_\beta(\Omega_2)} + \| \p_\alpha \psi^* \|_{ L^q_\alpha L^2_\beta(\Omega_2)} + \| \p_\alpha \psi^* \|_{L^1_\alpha H^s_\beta(\Omega_2)} $$
and
            $$ \| \phi \|_{X^{s}} = \| \phi^* \|_{L^\infty_\beta H^s_\alpha(\Omega_2)}+ \| \p_\beta \phi^* \|_{ L^q_\beta L^2_\alpha(\Omega_2)} + \| \p_\beta \phi^* \|_{L^1_\beta H^s_\alpha(\Omega_2)}.$$
It is easy to see that if $f \in H^s$ then solutions to
        \begin{align*} \p_t \psi + \p_x \psi &= 0 \\
                                            \psi(0)&= f \end{align*}
lie in $Y^{s}$ for any $q$. A similar comment applies for the space $X^{s}$. Furthermore, we have the following properties.

\begin{proposition}\label{persist emb}
Let  $0<s<\frac{1}{2}$.
\begin{enumerate}
\item For any $0<T<2$ we have the embeddings
                $$\| \psi \|_{L^\infty_t H^s_x(I_T \times I_{2-T} ) } \lesa \| \psi \|_{Y^{s}}$$
      and
                $$ \| \phi \|_{L^\infty_t H^s_x(I_T \times I_{2-T})} \lesa \| \phi \|_{X^{s}}.$$

\item Suppose $\p_t \psi + \p_x \psi = F$ with $\psi(0) = f$ and $f, F \in C^\infty_0$. Then
            $$ \| \psi \|_{Y^{s}} \lesa \| f\|_{H^s(I_2)} + \| F^* \|_{L^q_\alpha L^2_\beta(\Omega_2)} + \| F^* \|_{L^1_\alpha H^s_\beta(\Omega_2)}.$$
      Similarly, if $\p_t \phi - \p_x \phi = G$ with $\phi(0) = g$ and $g, G \in C^\infty_0$, then
            $$ \| \phi \|_{X^{s}} \lesa \| g\|_{H^s(I_2)} + \| G^* \|_{L^q_\beta L^2_\alpha(\Omega_2)} + \| G^* \|_{L^1_\beta H^s_\alpha(\Omega_2)}.$$
\end{enumerate}
\begin{proof}
We begin by proving $(i)$. Write
        $$\psi(t, x) = \int_0^{x+t} \p_\alpha \psi^*(\gamma, x-t) d\gamma + \psi^*(0 , x-t).$$
Since $(t, x) \in I_T \times I_{2-T}$ we have $|x-t| < 2$ and so
       $$ \| \psi^*(0, x-t)\|_{H^s_x(I_{2-T})} \les \| \psi^*(0, \beta) \|_{H_\beta^s(I_{2})}.$$
Together with a slight modification of Proposition \ref{int est} we see that
        \begin{align*}
            \| \psi(t, x) \|_{H^s_x(I_T)} &\les \Big\| \int_0^{x+t} \p_\alpha \psi^*(\gamma, x-t) d\gamma \Big\|_{H^s_x(I_{2-T})} + \| \psi^*(0, x-t) \|_{H^s_x(I_{2-T})} \\
            &\les \| \p_\alpha \psi^*(\alpha, x-t) \|_{L^1_\alpha H^s_x(I_2 \times I_{2-T} )} + \| \psi^*(0, \beta) \|_{H^s_\beta(I_2)} \\
            &\les \| \psi\|_{Y^{s}}.
        \end{align*}
The proof of the remaining estimate in $(i)$ is similar.

To prove $(ii)$ we follow the proof of Theorem \ref{energy} and write the solution $\psi$ as
            $$ \psi(t, x) = \psi^*(\alpha, \beta) = f(\beta) + \frac{1}{2} \int_\beta^\alpha F^*(\gamma, \beta) d\gamma.$$
Applying Proposition \ref{int est} we obtain $(ii)$ for $\psi$. The inequality for $\phi$ is similar and we omit the details.

\end{proof}
\end{proposition}

We will also need the following version of the decomposition of Proposition \ref{decomp}.

\begin{lemma}\label{persist est}
Assume  $f, g \in C^\infty_0$ and $|m|\les 1$. Let $(\psi, \phi) \in C^\infty$ be the corresponding solution to (\ref{Dirac}). Then we can write $(\psi, \phi) = (\psi_L, \phi_L) + (\psi_N, \phi_N)$ with $$|\psi_L(t, x) |= |f(x-t)|, \qquad  |\phi_L(t, x)|  = |g(x+t)|,$$ and $(\psi_N, \phi_N)$ satisfies
                    $$ \| \psi^*_N \|_{L^\infty_{\alpha, \beta}(\Omega_2)} + \| \phi_N^* \|_{L^\infty_{\alpha, \beta}(\Omega_2)}
                                \lesa \| \psi \|_{Y_2} + \| \phi\|_{X_2}.$$

\begin{proof}
We begin by using the same decomposition as in Proposition \ref{decomp}, $$(\psi, \phi) = (\psi_L, \phi_L) + (\psi_N, \phi_N),$$ where we recall that
 $|\psi_L^*(\alpha, \beta)| = |f(\beta)|$, $|\phi_L^*(\alpha, \beta)| = |g(\alpha)|$, and
    $$|\psi_N^*(\alpha, \beta)|^2 = m \int_\beta^\alpha \Im (\phi \overline{\psi}_N)^*(\gamma, \beta) d\gamma,
    \qquad |\phi_N^*(\alpha, \beta)|^2 = m \int_\beta^\alpha \Im (\psi \overline{\phi}_N)^*(\alpha, \gamma) d\gamma.$$
Estimating $\psi_N$ we get
        $$ \|\psi_N^*(\alpha)\|_{L^\infty_\beta(I_2) }^2 \les \| \phi^* \|_{L^2_\alpha L^\infty_\beta(\Omega_2)}^2 + \int_{-2}^\alpha \|\psi_N^*(\gamma)\|_{L^\infty_\beta(I_2)}^2 d\gamma$$
and so the estimate for $\psi_N$ follows by an application of Gronwall's inequality together with Lemma \ref{est}. The estimate for $\phi_N$ is similar.
\end{proof}
\end{lemma}

 The technical details are in place and we are now able to prove a local version of Theorem \ref{local bound}.

\begin{theorem}\label{local persist}

Let $0<s<\frac{1}{4}$ and $\frac{1}{2} = \frac{1}{p} + s$. There exists $0<\epsilon^*<1$ such that for any $|m|<\epsilon^*$ and $f, g \in C^\infty_0$ with
        $$ \| f\|_{L^p(\RR)} + \| g \|_{L^p(\RR)} < \epsilon^*,$$
the corresponding solution $(\psi, \phi ) \in C^\infty$ to (\ref{Dirac}) satisfies
        $$ \| \psi \|_{L^\infty_t H^s_x( I_{1} \times I_{1} )} + \| \phi \|_{L^\infty_t H^s_x( I_{1} \times I_{1})} \lesa \|f \|_{H^s(I_2)} + \| g\|_{H^s(I_2)}$$
with constant independent of $f$, $g$, and $m$.
\begin{proof}
By Proposition \ref{persist emb} it suffices to prove
                $$ \| \psi \|_{Y^{s}} + \| \phi \|_{X^{s}} \lesa \|f \|_{H^s(I_2)} + \| g\|_{H^s(I_2)} $$
Since $p>2$, by taking $\epsilon^*< \epsilon$, where $\epsilon$ is the constant appearing in Theorem \ref{localised}, we have a smooth solution\footnote{Note that the classical smooth solution from the initial data $f, g$ belongs to $Y_2 \times X_2$ and so agrees with the solution given by Theorem \ref{localised}. } $(\psi, \phi)$ such that
            $$ \| \psi \|_{Y_2} + \|\phi \|_{X_2}  < 2\epsilon^*. $$
An application of Proposition \ref{persist emb} gives the estimate
        \begin{equation}\label{local persist eqn1} \| \psi \|_{Y^s} \lesa \| f\|_{H^s(I_2)}  + \| m \phi^*  + 2 \lambda |\phi^*|^2 \psi^* \|_{L^q_\alpha L^2_\beta(\Omega_2)} + \| m \phi^*  + 2\lambda |\phi^*|^2 \psi^* \|_{L^1_\alpha H^s_\beta(\Omega_2)}. \end{equation}
For the second term, noting that $2q=p$ and $1<q<2$, we have by Lemma \ref{persist est}
        \begin{align*}
      \| m \phi^*  + \lambda |\phi^*|^2 \psi^* \|_{L^q_\alpha L^2_\beta(\Omega_2)}   &\lesa
               \epsilon^* \|  \phi^* \|_{L^2_\alpha L^\infty_\beta(\Omega_2)} +  \big\| |\phi^*|^2 \big\|_{L^{q}_\alpha L^\infty_\beta(\Omega_2)} \| \psi \|_{L^\infty_\alpha L^2_\beta(\Omega_2)} \\
      &\lesa \epsilon^* \| \phi \|_{X_2} +  \Big( \| g \|_{L^{2q} (I_2) }^2
                        + \| \phi_N \|^2_{L^\infty_{\alpha,\beta}(\Omega_2)} \Big) \| \psi \|_{Y_2} \\
      &\lesa  \epsilon^* \| \phi \|_{X_2} + \Big( \| g \|_{L^{p} (I_2) }^2 + \| \phi \|^2_{X_2} \Big) \| \psi \|_{Y_2}. \end{align*}
Thus taking $\epsilon^*>0$ sufficiently small, we have
        $$  \| m \phi^*  + \lambda |\phi^*|^2 \psi^* \|_{L^q_\alpha L^2_\beta(\Omega_2)} \les \frac{1}{8} \big( \| \psi \|_{Y^s} + \|\phi\|_{X^s} \big). $$
For the third term in (\ref{local persist eqn1}) we need to estimate $\| \phi^* \|_{L^1_\alpha H^s_\beta(\Omega_2)}$ and $\big\| |\phi^*|^2 \psi^*\big\|_{L^1_\alpha H^s_\beta(\Omega_2)}$. We can control the first term by using $(ii)$ in Proposition \ref{int est} while for the cubic term by Theorem \ref{Intrins} we have for all $\alpha \in I_2$
        $$ \big\| |\phi^*|^2 \psi^* \big\|_{H^s_\beta(I_2)} \lesa \| \phi^*\|_{L^\infty_\beta(I_2)}^2 \| \psi^*\|_{H^s_\beta(I_2)} + \| \phi^*\|_{L^\infty_\beta(I_2)} \| \phi^* \|_{ W^{1, q}_\beta(I_2)} \| \psi^* \|_{L^p_\beta(I_2)}.$$
Therefore,
        \begin{align*}
            \big\| |\phi^*|^2 \psi^* \big\|_{L^1_\alpha H^s_\beta(\Omega_2)} &\lesa \| \phi^* \|^2_{L^2_\alpha L^\infty_\beta(\Omega_2)} \| \psi^* \|_{L^\infty_\alpha H^s_\beta(\Omega_2)} + \| \phi^* \|_{L^2_\alpha L^\infty_\beta(\Omega_2)} \| \phi^*\|_{L^2_\alpha W^{1, q}_\beta(\Omega_2)} \| \psi^* \|_{L^\infty_\alpha L^p_\beta(\Omega_2)} \\
                    &\lesa \| \phi \|_{X_2}^2 \| \psi \|_{Y^s} + \|\phi \|_{X_2} \|\phi \|_{X^s} \big( \|f \|_{L^p(I_2)} + \|\psi \|_{Y_2}\big)\\
                    &\lesa \big(\epsilon^*\big)^2 \big( \| \psi \|_{Y^s} + \| \phi \|_{X^s} \big)
        \end{align*}
where we used Lemma \ref{persist est} together with the characterisation
        $$ \| f \|_{W^{1, p}(I_2)}  \backsimeq  \| f\|_{L^p(I_2)} + \| \p_x f \|_{L^p(I_2)}$$
which follows from the proof of Theorem \ref{Intrins}.
Thus provided we choose $\epsilon^*$ sufficiently small, we obtain
            $$\| \psi \|_{Y^s}  \les C \| f\|_{H^s(I_2)} + \frac{1}{4} \big( \| \psi \|_{Y^s} + \| \phi \|_{X^s} \big).$$
A similar argument shows
            $$\| \phi \|_{X^s}  \les C \| g\|_{H^s(I_2)} + \frac{1}{4} \big( \| \psi \|_{Y^s} + \| \phi \|_{X^s} \big)$$
and so result follows.

\end{proof}
\end{theorem}

Finally we come to the proof of Theorem \ref{local bound}.

\begin{proof}[Proof of Theorem \ref{local bound}]
Let $\epsilon^*>0$ be the constant from Theorem \ref{local persist}. Assume $f, g \in H^s$ satisfy (\ref{local bound eqn1}) and $|m|<\epsilon^*$. Choose a smooth approximating sequence $f_n, g_n \in C^\infty_0$ converging to $f, g$ in $H^s$. Note that we may also assume $f_n, g_n$ satisfy (\ref{local bound eqn1}) for every $n \in \NN$. Suppose for the moment that we had the estimate
            \begin{equation}\label{local bound eqn2} \| \psi_n \|_{L^\infty_t H^s_x (I_{1} \times \RR)} + \|\phi_n\|_{L^\infty_t H^s_x(I_{1} \times\RR)} \lesa \|f_n \|_{H^s(\RR) } + \|g_n\|_{H^s(\RR)}  \end{equation}
with the constant independent of $n\in \NN$. The continuous dependence on initial data proven by Selberg and Tesfahun in \cite{Selberg2010b} implies that $(\psi_n, \phi_n)$ converges to a solution $(\psi, \phi) \in C([-T^*, T^*], H^s)$ with possibly\footnote{The proof of local existence contained in \cite{Selberg2010b} gives a time of existence $T^*$ depending on the size of the initial data in $H^s$.} $T^*<1$. The uniform bound (\ref{local bound eqn2}) on the interval $[-1, 1]$ implies we can repeat the $H^s$ local existence result of \cite{Selberg2010b} and extend the solution to (at least) the interval $[-1, 1]$. Thus we obtain $(\psi , \phi ) \in C([-1,1], H^s)$ as required.

It remains to prove (\ref{local bound eqn2}). To this end assume $f, g \in C^\infty_0$ and let $(\psi, \phi)$ be the corresponding smooth solution. Similar to the proof of Theorem \ref{local} we take $I_R(x) = [ x - R, x+R]$. Since the Dirac equation is invariant under translation by Theorem \ref{local persist} we have for every $j \in \ZZ$
        $$ \| \psi \|_{L^\infty_t H^s_x( I_1 \times I_1(j) )} + \| \phi \|_{L^\infty_t H^s_x(I_1 \times I_1(j))} \lesa
        \| f\|_{H^s(I_2(j))} + \| g \|_{H^s(I_2(j))}.$$
Therefore, by  $(ii)$ in Theorem \ref{Intrins} we have for every $|t|\les 1$
            \begin{align*}
                \| \psi(t) \|_{H^s_x}^2 + \| \phi(t) \|_{H^s_x}^2 &\lesa \sum_{j \in \ZZ} \Big( \| \psi(t) \|_{H^s_x(I_{1}(j))}^2 + \| \phi(t) \|_{H^s_x(I_{1}(j))}^2 \Big)\\
                                                              &\lesa \sum_{j \in \ZZ}  \Big( \| f\|_{H^s(I_2(j))}^2 + \| g\|_{H^s(I_2(j))}^2\Big)\\
                                                              &\lesa \| f\|_{H^s}^2 + \|g\|_{H^s}^2.\end{align*}
 Thus the inequality  (\ref{local bound eqn2}) holds and the result follows.

\end{proof}

\appendix
\section{Intrinsic Definition of $H^s(I)$}

 For the convenience of the reader, sketch the proof of Theorem \ref{Intrins}. The inequalities are all well known, see for instance \cite{Tartar2007} page 169 for a  more general version\footnote{We should mention that the $W^{s, p}$ spaces defined in \cite{Tartar2007} do not agree with the Sobolev spaces defined in the present paper. More precisely  in \cite{Tartar2007} the author takes $W^{s, p} = B^s_{p, p}$, where $B^s_{p, q}$ is the Besov-Lipschitz space.  Thus the definitions only agree in the case $p=2$.} of $(i)$. Part $(ii)$ is essentially a corollary of $(i)$ while the last inequality $(iii)$ is an application of Littlewood-Paley Theory.

\begin{proof}[Proof of Theorem \ref{Intrins}]
We start by proving  $(i)$. For an interval $I_R=[-R, R]$ define
        $$ \| f\|_{\widetilde{H^s}(I_R)}^2 = \| f\|_{L^2(I_R)}^2 + \int_{I_R} \int_{I_R} \frac{| f(x) - f(y) |^2}{|x-y|^{1+2s} }dy dx.$$
We need to show that
            \begin{equation}\label{Intrins eqn1} \| f\|_{H^s(I_R)} \lesa_R \| f\|_{\widetilde{H^s}(I_R)} \lesa_R \| f\|_{H^s(I_R)}. \end{equation}
 The second inequality is straight forward as take any extension $g \in H^s(\RR)$ of $f \in H^s(I_R)$. Then
        \begin{align*}
         \| f\|_{\widetilde{H^s}}^2
                &\lesa \| g\|_{L^2(\RR)}^2 + \int_{\RR^n} \int_{\RR^n} \frac{|g(x ) - g(y) |^2}{|x-y|^{1+2s}} dxdy \\
                        &\lesa \| g\|_{H^s}^2
         \end{align*}
and so taking the infimum over all extensions $g$, we obtain the second inequality in (\ref{Intrins eqn1}).

The first inequality in (\ref{Intrins eqn1}) is harder to prove. For $f \in L^2(I_R)$ define an extension $\mathcal{E}(f)$ of $f$ by letting,
            \begin{equation}\label{Intrins eqn2}
                \mathcal{E}(f)(x) = \begin{cases} \rho(x) f( \pm 2 R - x) \qquad \qquad & \pm x \g R \\
                                        f(x)                & |x|<R \end{cases} \end{equation}
where $\rho \in C^\infty_0$ with $\rho(x) = 1$ for $|x|<R$, $\supp \rho \subset\{ |x|<2R\}$, and $|\rho(x)|\les 1$ for ever $x \in \RR$. Then
        \begin{align*}
            \| f\|_{H^s(I_R)}^2 \les \|\mathcal{E}(f)\|_{H^s(\RR)}^2 &\lesa \|\mathcal{E}(f) \|_{L^2(\RR)}^2 + \int_\RR \int_\RR \frac{|\mathcal{E}(f)(x) - \mathcal{E}(f)(y) |^2}{ |x-y|^{1+2s}} dxdy \\
                                                      &\lesa \| f\|_{L^2(I_R)} + \int_{I_R} \int_{I_R} + 2\int_{I_R} \int_{ y>R} + 2\int_{I_R} \int_{ y<-R}  + \int_{\RR \setminus I_R} \int_{\RR \setminus I_R}. \end{align*}
The first integral term is obvious. For the second we note that for $|x|<R$ and $ y>R$ we have
            \begin{align*} |\mathcal{E}(f)(x) - \mathcal{E}(f)(y) | &= | f(x) - \rho(y) f(2R-y)| \\
                                                                    &\les |f(x)| |\rho(x) - \rho(y) | + |f(x) - f(2R-y)| \end{align*}
and $   | x - (2R-y)| \les |x - y|$. Hence
        \begin{align*}
        \int_{I_R} \int_{y>R} \frac{|\mathcal{E}(f)(x) - \mathcal{E}(f)(y) |^2}{ |x-y|^{1+2s}} dy dx &\lesa \int_{I_R} \int_{R}^{2R} \frac{|f(x)|^2 |\rho(x) - \rho(y) |^2}{|x-y|^{1+2s}} dy dx\\
                        &\qquad \qquad \qquad \qquad+ \int_{I_R} \int_{ R}^{2R} \frac{|f(x) - f(2R-y) |^2}{ |x-(2R-y)|^{1+2s}} dy dx \\
        &\lesa_{\rho} \| f\|_{L^2(I_R)}^2 \int_0^{3R} r^{1-2s} dr + \int_{I_R} \int_{ I_R } \frac{|f(x) - f(y) |^2}{ |x-y|^{1+2s}} dy dx \\
        &\lesa_{\rho, R} \| f\|_{\widetilde{H^s}(I_R)}^2 .\end{align*}
The other terms are handled similarly and so we obtain (\ref{Intrins eqn1}).

We now prove $(ii)$. The second inequality  follows from a simple application of $(i)$ and so we will concentrate on the first. Since
        $$ \| f\|_{H^s}^2 \lesa \| f\|^2_{L^2} + \int_\RR \int_\RR \frac{|f(x) - f(y) |^2}{ |x-y|^{1+2s}} dx dy $$
it suffices to prove
       \begin{equation}\label{local global eqn3} \int_\RR \int_\RR \frac{|f(x) - f(y) |^2}{ |x-y|^{1+2s}} dx dy \les \sum_{j \in \ZZ} \| f\|_{H^s(I_1(j))}^2 .\end{equation}
Now
    \begin{align*}
        \int_\RR \int_\RR \frac{|f(x) - f(y) |^2}{ |x-y|^{1+2s}} dx dy &= \sum_{j \in \ZZ} \int_{I_{\frac{1}{2}}(j)}  \int_\RR \frac{|f(x) - f(y) |^2}{ |x-y|^{1+2s}} dx dy \\
                        &= \sum_{j \in \ZZ} \int_{I_{\frac{1}{2}(j)}} \int_{I_1(j) }\frac{|f(x) - f(y) |^2}{ |x-y|^{1+2s}} dx dy \\
                        &\qquad \qquad \qquad \qquad \qquad + \int_{I_{\frac{1}{2}}(j)}    \int_{|x-j|>1} \frac{|f(x) - f(y) |^2}{ |x-y|^{1+2s}} dx dy.
    \end{align*}
Part $(i)$ allows us to control the first integral. For the second we have
    \begin{align*}
            \int_{I_{\frac{1}{2}}(j) } \int_{|x-j|>1} \frac{|f(x) - f(y) |^2}{ |x-y|^{1+2s}} dx dy
                                    &\lesa\int_{I_{\frac{1}{2}}(j)} |f(y)|^2 dy + \int_{ |x-j|>1} \frac{ |f(x)|^2 }{ \big( 2 |x-j| - 1\big)^{1+ 2s}} dx \\
                                    &\lesa \| f\|_{L^2(I_1(j))}^2 + \sum_{|j-k|\g1 } \frac{1}{|j-k|^{1+2s}} \| f\|_{L^2(I_1(k))}^2.
\end{align*}

Therefore, as $1+2s >1$, an application of Young's inequality for sequences gives (\ref{local global eqn3}) and so result follows.

Finally we come to the proof of $(iii)$. We begin by fixing $s<\frac{1}{q} <1$. An application of Littlewood-Paley theory\footnote{See for instance the appendix in \cite{Tao2006b}. } together with  Sobolev embedding shows that,
         \begin{align*} \| |g|^2 f \|_{H^s(\RR)}
                    &\lesa \| g\|_{L^\infty(\RR)}^2 \| f\|_{H^s(\RR)} + \| g\|_{L^\infty(\RR)} \| g\|_{W^{s, r}(\RR)} \| f\|_{L^p(\RR)} \\
                    &\lesa \| g\|_{L^\infty(\RR)}^2 \| f\|_{H^s(\RR)} + \| g\|_{L^\infty(\RR)} \| g\|_{W^{1, q}(\RR)} \| f\|_{L^p(\RR)}\end{align*}
where $r=\frac{1}{s}$. To obtain $(iii)$ we make use of the extension operator $\mathcal{E}$ defined above. It is easy to see that $\mathcal{E}$ is bounded on $L^r$ for every $1\les r \les \infty$. Moreover the proof of $(i)$ shows that it is also bounded  on $H^s$. Hence
        \begin{align*}
            \big\| |g|^2 f \big\|_{H^s(I_R)} &\les \big\| |\mathcal{E} (g) |^2 \mathcal{E} (f) \big\|_{H^s(\RR)} \\
                                      &\lesa \| \mathcal{E} (g) \|_{L^\infty(\RR)}^2 \| \mathcal{E} (f) \|_{H^s(\RR)} + \|\mathcal{E} (g)\|_{L^\infty(\RR)} \| \mathcal{E} (g)\|_{W^{1, q}(\RR)} \| \mathcal{E} (f) \|_{L^p(\RR)}\\
                                      &\lesa \| g \|_{L^\infty(I_R)}^2 \| f \|_{H^s(I_R)} + \| g\|_{L^\infty(I_R)} \| \mathcal{E}(g)\|_{W^{1, q}(\RR)} \| f\|_{L^p(I_R)} \end{align*}
and so it suffices to prove that
                $$ \|\mathcal{E} (g) \|_{W^{1, q}(\RR)} \lesa \| g\|_{W^{1, q}(I_R)}.$$
However this follows easily using the characterisation
                $$ \| f \|_{W^{1, p}(\RR)} \backsimeq \| f\|_{L^p(\RR)} + \| \p_x f \|_{L^p(\RR)}.$$
Note that as a consequence of this we have
            $$ \| f \|_{W^{1, p}(I_R)}  \backsimeq_R  \| f\|_{L^p(I_R)} + \| \p_x f \|_{L^p(I_R)}.$$

\end{proof}

\bibliographystyle{amsplain}
\bibliography{DiracPaper}

\end{document}